\documentclass[12pt,reqno]{amsart}
\usepackage[a4paper,margin=1in]{geometry}

\usepackage{microtype}
\usepackage[cal=euler,scr=boondoxo]{mathalfa}
\usepackage{tikz,tikz-cd}
\usetikzlibrary{decorations}
\usepackage[colorlinks,allcolors=blue]{hyperref}
\usepackage{graphicx}
\usepackage{xcolor}
\usepackage[capitalize]{cleveref}

\hypersetup{
  pdftitle={Multiple zeta values in lambda-rings},
  pdfauthor={Sergey Mozgovoy}
}

\input{commands}
\defcite\sw{sweedler_hopf}
\defcite\mac{macdonald_symmetric}
\defcite\bur{burmester_balanced}
\defcite\bura{burmeister_algebraic}

\opr\T{\mathsf T} %Tensor algebra
\opr\Tc{\T^{\mathrm c}}
\dmat\Tp{\T^+}

\def\R{\mathsf R}
\def\Rp{\R^+} %f(0)=0
\opr\qz{\mathsf{qZ}} %Algebra of qMZV
\def\qzp{\qz^+}
\def\qZ{\qz}

\oper{wt}

\dmat\QS{\mathsf{QS}}
\dmat\QSp{\ub\QS} %positive QS, free CTD algebra
\forcsvlist\oper{QSh,CTD,CQS,CRB,CB,Coalg,sh}

\dmat\C{\mathsf C}
\dmat\A{\mathsf A}

\oper{Li}

\oper{qMZ}
\oper{MD}
\oper{qMZV}

\opr\MT{\mathsf{MT}} %mixed Tate motives

\oper{CAA}
\oper{CBA}
\oper{Mon}

\opr\MT{\mathsf{MT}}
\oper{DM}

\oper{GRT}
\oper{GT}

\oper{Coalg}
\oper{Bialg}
\oper{Hopf}
\oper{MZV}
\oper{RB}
\def\CC{\opn{Coalg}_\con} %connected coalgebras
\oper{corad}
\oper{sh}
\oper{qsh}

\def\nun{\mathrm{nu}} %nonunital
\def\con{\mathrm{con}}
\def\SZ{\mathrm{SZ}}
\def\BZ{\mathrm{BZ}}
\def\sf{\mathsf f}
\def\reg{\mathrm{reg}}

\def\mzf{multiple zeta function\xspace}
\def\mzv{multiple zeta values\xspace}

\def\bW{\ubar W}
\def\bU{\ubar U}
\def\C{\ubar C} %augmentation ideal
\def\bde{\bar\de}

\def\Lt{L_t}
\def\bLt{\ubar L_t}
\def\y{z_0} %generator of the Harmonic algebra
\def\P{\cP}

\def\eH{\mathfrak H} %extended Harmonic algebra
\def\Zf{\cZ^\sf_*} %formal algebra of qMZV
\def\qz#1{\cZ_{q,#1}} %qmzv
\def\qzp#1{\cZ^+_{q,#1}} %qmzv^+
\def\qZ{\cZ}
\def\Zr{\cZ_\reg}
\def\Zrp{\cZ^+_\reg}
\def\Ap{A^+}
\def\QQ{\bQ\pser q}

\begin{document}
\title{Multiple zeta values in \tpdf{$\lambda$}{lambda}-rings}
\author{Sergey Mozgovoy}
\address{School of Mathematics, Trinity College Dublin, Dublin 2, Ireland
\newline\indent
Hamilton Mathematics Institute, Dublin 2, Ireland}
\email{mozgovoy@maths.tcd.ie}

\begin{abstract}
We propose a framework for multiple zeta values in $\lambda$-rings
that unifies both classical multiple zeta values and multiple $q$-zeta values. As in the classical case, we show that the multiple zeta function defines an algebra homomorphism from the quasi-shuffle algebra to the $\lambda$-ring. We also prove a formula for multiple polylogarithms in $\lambda$-rings, similar to the iterated integral expression in the classical case. As an application, we explicitly describe the algebra of regular multiple $q$-zeta values, compute the corresponding Poincar\'e series, and exhibit a small spanning set.
\end{abstract}

\maketitle
\section*{Introduction}
Multiple zeta values
\begin{equation}\label{ze1}
\ze(s_1,\dots,s_k)=\sum_{n_1>\dots>n_k>0}\prod_{i=1}^k\frac1{n_i^{s_i}},\qquad s\in\bZ_{\ge1}^k,\, s_1\ge2,
\end{equation}
were introduced by Euler and Goldbach \cite{euler_meditationes}
for $k=2$,
and have since appeared
in numerous areas of mathematics and theoretical physics
\cite[\S1.1]{goncharov_multiple}.
The $\bQ$-vector space $\cZ\sbs\bR$ spanned by \mzv
is a commutative algebra equipped with a weight filtration,
where $\ze(s_1,\dots,s_k)$ has weight $\sum_i s_i$.
It was conjectured in \cite{zagier_values}
that the Poincar\'e series of $\cZ$ is $\frac1{1-t^2-t^3}$.
This conjecture was strengthened in \cite{hoffman_algebra},
by proposing that the MZV $\ze(s_1,\dots,s_k)$ with $s_i\in\set{2,3}$ form a basis of $\cZ$.
It was proved in \cite{brown_mixed} that $\cZ$ is spanned by these elements.

The algebra structure of $\cZ$ was also investigated.
Let $U=\bop_{i\ge1}\bQ z_i$ be the non-unital algebra with $z_i\circ z_j=z_{i+j}$.
Then there is a surjective algebra morphism
\cite{hoffman_algebra,racinet_series,ihara_derivation}
\eq{\ze^*:\QS(U)=\bop_{k\ge0}U^{\ts k}\to\cZ,}
where $\QS(U)$ is the quasi-shuffle (Hopf) algebra of $U$
\cite{newman_cofree,hoffman_quasia,fares_quelques},
$\ze^*(z_1)=0$ and
\begin{equation}\label{ze-homom}
\ze^*(z_{s_1}\dots z_{s_k})=\ze(s_1,\dots,s_k),\qquad s_1\ge2.
\end{equation}
\oper{QSym}
The Hopf algebra $\QS(U)$ is isomorphic to the Hopf algebra $\QSym$ of quasi-symmetric functions by \cite[Theorem 3.4]{hoffman_algebra}.
The kernel of $\ze^*$ was conjecturally described in
\cite{racinet_series,ihara_derivation}.
\medskip

Similarly, given $f_1,\dots,f_k\in \Ap=q\QQ$,
consider the multiple $q$-zeta value
\eq{Z(f_1\ts\dots\ts f_k)
=\sum_{n_1>\dots>n_k>0}\prod_{i=1}^k f_i(q^{n_i})\in\QQ.
}
This extends to an algebra morphism
\[Z:\QS(\Ap)\to\QQ,\]
where $\QS(\Ap)$ is the quasi-shuffle algebra of the non-unital algebra $\Ap$.
In the literature on multiple $q$-zeta values
it is common to fix a sequence $\bff=(f_i)_{i\ge1}$ in~$\Ap$ and define
\[Z^\bff(s_1,\dots,s_k)=Z(f_{s_1}\ts\dots\ts f_{s_k}),\qquad
s\in\bZ_{\ge1}^k.
\]
If $g_i=(1-q)^i f_i\in\bQ[q]$ and $g_i(1)=1$ for all $i\ge1$, then
\eq{(1-q)^{\sum_i s_i}Z^\bff(s_1,\dots,s_k)\big|_{q\to1}
=\ze(s_1,\dots,s_k),}
for $s_1\ge2$.
Moreover, if $\deg g_i\le i$ for all $i\ge1$, then $\bff$ is a basis of $\Rp_0=\R_0\cap \Ap$ (see \S\ref{sec:qMZV}), where
\[
\R_0=\bQ\sbig{\tfrac q{1-q}}
=\sets{f\in\bQ\sbig{q,\tfrac1{q-1}}}{\deg f\le0}=\cO(\bP_\bQ^1\ms\set1).\]
Therefore the image
\[\qzp0=\Im(Z:\QS(\Rp_0)\to\QQ)\]
is spanned by the elements $Z^\bff(s_1,\dots,s_k)$ for $s\in\bZ_{\ge1}^k$.
Let $f_0=1$.
The algebra $\qz0$ spanned by the elements $Z^\bff(s_1,\dots,s_k)$
for $s\in\bN^k$ with $s_1\ge1$ is equal to $\qzp0$ by \cite{hirose_unified}.
It is conjectured in \cite{bachmann_dimension}
(see also \cite{okounkov_hilbert})
that $\qzp0$ is spanned by $Z^\bff(s_1,\dots,s_k)$ with $s_i\in\set{1,2,3}$ (although they don't form a basis).
There are also conjectures for the Poincar\'e series of $\qzp0$ and for the kernel of the map $Z:\QS(\Rp_0)\to\qzp0$
(see \S\ref{sec:qMZV} and \S\ref{sec:duality}).
% (see Conjectures \ref{conj1}-\ref{conj2}).

The above formulas for \mzv suggest that a similar formalism can be developed for arbitrary \la-rings (see \eg \cite{atiyah_group}). % (see Appendix \ref{lambda}).
Given a \la-ring $A$ with Adams operations $\psi^n$,
we define the truncated multiple zeta function for $n\ge1$ as
\[Z_n:\QS(A)\to A,\qquad
a_1\ts\dots\ts a_k\mto \sum_{n>n_1>\dots>n_k>0}\prod_{i=1}^k \psi^{n_i}(a_i),
\]
and we show that $Z_n$ is an algebra morphism.
This result extends to untruncated multiple zeta functions when the $\la$-ring $A$ is complete.
We will see that both classical \mzv and multiple $q$-zeta values can be interpreted within this framework.
In the case of multiple $q$-zeta values, one considers the \la-ring $\QQ$ with Adams operations $\psi^n(f)=f(q^n)$.

For the computations of \mzv, we introduce the multiple polylogarithm
\begin{equation*}
\Lt:\QS(A)\to A\pser t,\qquad
\Lt(a_1\ts\dots\ts a_k)
=\sum_{n_1>\dots>n_k>0}t^{n_1}\prod_{i=1}^k \psi^{n_i}(a_i),
\end{equation*}
which is an analogue of the classical multiple polylogarithm
\[\Li_{s_1,\dots,s_k}(t)=\sum_{n_1>\dots>n_k>0}t^{n_1}\prod_{i=1}^k\frac1{n_i^{s_i}}.\]
The following result is an analogue of the iterated integral formula for multiple zeta values \cite{zagier_values,hoffman_algebra}
and a formula for multiple $q$-zeta values \cite{singer_q}
(see Remark \ref{iter}).

\begin{customthm}{Theorem 1}
Consider continuous $A$-linear operators
\[\P_a:tA\pser t\to tA\pser t,\qquad
t^n\mto \psi^n(a)t^n,\qquad
\cT:A\pser t\to tA\pser t,\qquad f\mto\frac t{1-t}f,\]
for $a\in A$.
Then
\[\Lt(a_1\ts\dots\ts a_k)=\P_{a_1}\cT\dots \P_{a_k}\cT(1).\]
\end{customthm}

As an application, we will study the algebra of multiple zeta values
\[\Zrp=\Im\rbr{Z:\QS(\bQ[q]^+)\to \QQ}\]
associated with the algebra $\bQ[q]=\cO(\bP^1_\bQ\ms\set\infty)$
in place of the algebra $\R_0=\cO(\bP^1_\bQ\ms\set1)$ considered earlier.
It is spanned by the elements
\[Z^\bff(s_1,\dots,s_k)=Z(q^{s_1}\ts\dots\ts q^{s_k}),\qquad
s\in\bZ_{\ge1}^k.\]
A closely related algebra $\Zr$
is spanned by the elements
$Z^\bff(s_1,\dots,s_k)$ for $s\in\bN^k$ with $s_1\ge1$.
Both algebras carry weight filtrations
with
$\wt Z^\bff(s_1,\dots,s_k)\le \sum_i\max\set{s_i,1}$.

We will show that the algebras $\Zrp$ and $\Zr$ coincide,
but their weight filtrations are different.
This is in contrast with the algebras $\qzp0$ and $\qz0$,
for which the equality of the weight filtrations can be deduced
from the results of \cite{hirose_unified}.
More precisely, we prove the following.

\begin{customthm}{Theorem 2}
We have
\begin{enumerate}
\item
The algebras $\Zrp$ and $\Zr$ coincide
and are generated by the elements $\frac{q^k}{1-q^k}$ for $k\ge1$.
\item
The Poincar\'e series of $\Zrp$ is
\[
P(\Zrp,t)
=\sum_{i\ge0}\dim\gr^W_i(\Zrp)t^i
=1+\sum_{d\ge1}dt^d=\frac{1-t+t^2}{(1-t)^2}.\]

\item
The Poincar\'e series of $\Zr$ is
\[
P(\Zr,t)
=\sum_{i\ge0}\dim\gr_i^W(\Zr)t^i
=1+\frac{1}{1-t}\sum_{d\ge1}\phi(d)t^d,
\]
where $\phi$ is Euler's totient function.

\item
The algebra
$\Zrp$ is spanned by the elements $Z^\bff(s_1,\dots,s_k)$
with $s_i\in\set{1,2}$.
\end{enumerate}
\end{customthm}

\subsection*{Acknowledgments}
The author would like to thank Francis Brown, Vladimir Dotsenko,
Adam Keilthy and Florian Naef for many useful discussions.
He is also grateful to Henrik Bachmann and Annika Burmester for their helpful comments on the manuscript.
\section{Quasi-shuffle algebras} \label{sec:QSA}
%\subsection{Quasi-shuffle algebras}\label{qs-algebras}
Quasi-shuffle (Hopf) algebras were introduced in \cite{newman_cofree}
under the name \idef{cofree irreducible Hopf algebras}.
They were independently discovered in \cite{hoffman_quasia,fares_quelques}.
Let $K$ be a field and let $\Hopf_{\con}$ be the category of connected Hopf algebras over $K$
(also called conilpotent or irreducible Hopf algebras).
Let $\Ass_\nun$ be the category of associative non-unital algebras over $K$.
It was proved in \cite{newman_cofree} that
the forgetful functor
\[F:\Hopf_{\con}\to\Ass_\nun,\qquad (H,\mu,\eta,\de,\eps)\mto \bar H=\Ker(\eps),\]
has a right adjoint
\[\QS:\Ass_\nun\to\Hopf_{\con},\qquad
(A,\circ)\mto \QS(A,\circ)=\bop_{n\ge0}A^{\ts n}.\]
The Hopf algebra $\QS(A)=\QS(A,\circ)$
is called the \idef{quasi-shuffle algebra} of $A$.
It is equipped with an associative product $*$,
called the \idef{quasi-shuffle product},
defined inductively by $1_K*u=u*1_K=u$ and \cite{hoffman_quasia}
\begin{equation}\label{qs-prod1}
(au)*(bv)=a(u*bv)+b(au*v)+(a\circ b)(u*v),\qquad
a,b\in A,\, u,v\in \QS(A),
\end{equation}
where we write $uv$ instead of $u\ts  v$.
It is also equipped with the coproduct and counit of the cofree coalgebra
%(we abbreviate $u\ts v$ by $uv$),
\begin{equation}\label{coprod1}
\de(a_1\dots a_n)=\sum_{i=0}^n a_1\dots a_i\ts a_{i+1}\dots a_n,
\qquad
\eps:\QS(A)\to A^{\ts0}=K.
\end{equation}
With these structures, $\QS(A)$ becomes a conilpotent bialgebra.
It is automatically a Hopf algebra by \sw[9.2.2].
Below we will discuss the properties of $\QS(A)$ in more detail.

Define the \idef{shuffle algebra} $\Sh(A)=\QS(A,0)$,
corresponding to the zero product on $A$.
Its product is denoted by $\sha$ and is called the \idef{shuffle product}.
If $A$ is a commutative algebra over a field of characteristic zero,
then $\QS(A,\circ)$ is canonically isomorphic to $\Sh(A)$
by \cite[Theorem 1.12]{newman_cofree}
and \cite[Theorem 3.3]{hoffman_quasia}.

\subsection{Combinatorial description of the quasi-shuffle product}
The inductive formula \eqref{qs-prod1} for the quasi-shuffle product
admits several equivalent combinatorial descriptions,
in terms of quasi-shuffles (\cf \cite{cartier_structure}),
mixable shuffles (\cf \cite{newman_cofree,guo_baxter}),
and lattice paths \cite{fares_quelques,aguiar_monoidal}.

For a surjective map $\si:[m]\to[n]=\set{1,\dots,n}$, let
\begin{equation}
\si_*:A^{\ts m}\to A^{\ts n},\qquad
a_1\ts\dots\ts a_m\mto c_1\ts\dots\ts c_n,\qquad c_j=\prod_{i\in\si\inv(j)}a_i,
\end{equation}
where the product is taken in increasing order of $i$.
Define the set of \idef{quasi-shuffles}
$\qsh(m,n)=\bigsqcup_{r\ge0}\qsh(m,n;r)$,
where $\qsh(m,n;r)$ consists of surjective maps
$\si:[m+n]\to[r]$ satisfying
\begin{equation}\label{qs1}
\si(1)<\dots<\si(m),\qquad \si(m+1)<\dots<\si(m+n).
\end{equation}
In particular, the set $\sh(m,n)=\qsh(m,n;m+n)$ consists of permutations and is called the set of \idef{shuffles}.

We can identify $\si\in\qsh(m,n;r)$ with the pair
$(I,J)=(\si([m]),\si(m+[n]))$
of subsets of $[r]$ satisfying $[r]=I\cup J$.
For $u=a_1\ts\dots\ts a_m$ and $v=b_1\ts\dots\ts b_n$,
we have
\begin{equation}
\si_*(u\ts v)=c_1\ts\dots \ts c_r,\qquad\qquad
c_k=\begin{cases}
a_i & k=\si(i)\notin\si(m+[n]),\\
b_j & k=\si(m+j)\notin \si([m]),\\
a_i\circ b_j & k=\si(i)=\si(m+j).
\end{cases}
\end{equation}
The inductive formula \eqref{qs-prod1} is equivalent to
\begin{equation}\label{qs-prod3}
u*v=\sum_{\si\in\qsh(m,n)}\si_*(u\ts v).
\end{equation}

\begin{remark}
There is an alternative way to parametrize quasi-shuffles, which was used in \cite{newman_cofree} to express the quasi-shuffle product.
Given a quasi-shuffle $\si\in\qsh(m,n;r)$,
let $T$ be the set of pairs $(i,j)\in[m]\xx(m+[n])$ such that $\si(i)=\si(j)$.
Define the total order on $[m+n]$, where $i\prec j$ if either $\si(i)<\si(j)$ or $(i,j)\in T$.
This total order corresponds to a permutation $\ta\in\sh(m,n)\sbs S_{m+n}$
such that $\ta (i)<\ta (j)$ for $i\prec j$.
Moreover, $\ta(i)+1=\ta(j)$ for $(i,j)\in T$.
The pair $(\ta,T)$ is called a mixable shuffle (\cf \cite{guo_baxter}).
It determines the quasi-shuffle~$\si$ uniquely.
Indeed, consider the equivalence relation on $[m+n]$ generated by $\ta(i)\sim\ta(i)+1$ for $(i,j)\in T$.
Then the quotient $[m+n]\qt\sim$ can be identified with $[r]$ and the
composition $[m+n]\xto\ta[m+n]\to[r]$ coincides with $\si$.
\end{remark}

Another equivalent formula for the quasi-shuffle product uses lattice paths.
Let $S=\set{(1,0),(0,1),(1,1)}\sbs\bN^2$ and let $S^*=\bigsqcup_{n\ge0}S^n$ be the free monoid generated by~$S$.
Consider the map
\[\pi:S^*\to\bN^2,\qquad
s_1\dots s_n\mto\sum*_i s_i,\]
and define the set of \idef{Delannoy paths} from $(0,0)$ to $(m,n)$
to be $\cD(m,n)=\pi\inv(m,n)$.
For $\ell\in\cD(m,n)$, let $\ga_\ell:A^{\ts m}\xx A^{\ts n}\to \QS(A)$ be defined as follows.
If $m=0$ or $n=0$, let $\ga_\ell(u,v)=uv$.
If $\ell=s\ell'\in \cD(m,n)$ with $s\in S$
and $u=a\ts u'\in A^{\ts m}$, $v=b\ts v'\in A^{\ts n}$ with $a,b\in A$, let
\begin{equation}
\ga_\ell(u,v)
=\begin{cases}
a\ts\ga_{\ell'}(u',v)&s=(1,0),\\
b\ts\ga_{\ell'}(u,v')&s=(0,1),\\
(a\circ b)\ts\ga_{\ell'}(u',v')&s=(1,1).
\end{cases}
\end{equation}

\begin{lemma}
The quasi-shuffle product
of $u\in A^{\ts m}$ and $v\in A^{\ts n}$ is given by
\begin{equation}\label{qs-prod2}
u*v=\sum_{\ell\in\cD(m,n)}\ga_\ell(u,v).
\end{equation}
\end{lemma}
\begin{proof}
To see that \eqref{qs-prod3} is equivalent to \eqref{qs-prod2},
we note that there is a bijection between the set of length~$r$ paths in $\cD(m,n)$ and the set of quasi-shuffles $\qsh(m,n;r)$.
For a path $\ell=s_1\dots s_r\in\cD(m,n)$, we consider $\si\in\qsh(m,n;r)$ corresponding to the pair $(I,J)$, where
\[I=\sets{i\in[r]}{s_i=(1,0)\text{ or }(1,1)},\qquad
J=\sets{i\in[r]}{s_i=(0,1)\text{ or }(1,1)}.\]
We have $(m,n)=\sum_i s_i=(\n I,\n J)$ and $\si_*(u\ts v)=\ga_\ell(u,v)$.
\end{proof}

\subsection{Conilpotent coalgebras}
A coalgebra $(C,\de,\eps)$ over a field $K$ is called \idef{coaugmented} if it is equipped with a morphism of coalgebras $\eta:K\to C$. %, called a coaugmentation.
Equivalently, $C$ is equipped with a group-like element $1_C=\eta(1_K)\in C$, meaning that $\de(1_C)=1_C\ts 1_C$ and $1_C\ne0$.
We have $\eps\eta=\id$ and $C=K1_C\oplus \bar C$, where $\bar C=\Ker\eps$.
The reduced coproduct
\[\bar\de:\bar C\xto\de C\ts C\to\bar C\ts\bar C,\qquad
\bar\de(x)=\de(x)-x\ts1_C-1_C\ts x,\]
is coassociative (note that it depends on the coaugmentation).
We define $\de_n:C\to C^{\ts n}$
with $\de_1=\id$ and $\de_{n}=(\id\ts\de_{n-1})\de$ for $n\ge2$.
%(we can also define $\de_0=\eps$)
Similarly, define $\bar\de_n:\bar C\to\bar C^{\ts n}$ for $n\ge1$.

\begin{lemma}
We have
\begin{equation}\label{de-bde}
\de_n=\sum_{\ov{\si:[m]\emb[n]}{m\ge1}}\si_*\bar\de_m,\qquad
\bar\de_n=\sum_{\ov{\si:[m]\emb[n]}{m\ge1}}(-1)^{n-m} \si_*\de_m
\end{equation}
on $\bar C$, where the sums run over strictly increasing maps $\si:[m]\emb[n]$ and
\[\si_*:C^{\ts m}\to C^{\ts n},\qquad
a_1\ts\dots\ts a_m\mto b_1\ts\dots b_n,\qquad
b_j=\begin{cases}
a_i&\si(i)=j,\\
1&j\notin \si([m]).
\end{cases}\]
\end{lemma}
\begin{proof}
The proof is left to the reader.
\end{proof}

The following result highlights the relevance of quasi-shuffles in the theory of bialgebras.

\begin{lemma}
\label{qs coprod}
If $B$ is a bialgebra and $a,b\in\bar B$, then
\[\bar\de_m(ab)
=\sum_{i,j\ge1}\sum_{\si\in\qsh(i,j;m)}\si_*(\bar\de_i(a)\ts\bar\de_j(b)).\]
\end{lemma}
\begin{proof}
By \eqref{de-bde}, it is enough to show that for every $n\ge1$
\[\sum_{\ov{\ta:[m]\emb[n]}{m\ge1}}\ta_*
\sum_{\ov{\si\in\qsh(i,j;m)}{i,j\ge1}}\si_*(\bar\de_i(a)\ts\bar\de_j(b))
=\de_n(ab).
\]
The left hand side is
$
\sum_{i,j\ge1}
\sum_{\si'\col[i]\emb[n]}\sum_{\si''\col[j]\emb[n]}
\si'_*\bar\de_i(a)\cdot \si''_*\bar\de_j(b)
=\de_n(a)\cdot \de_n(b)=\de_n(ab).$
\end{proof}

\begin{lemma}\label{coaugm filtr}
Let $C$ be a coaugmented coalgebra, $C_0=K1_C$ and
\[C_{n+1}=C_n\wedge C_0=\sets{x\in C}{\de(x)\in C_n\ts C+C\ts C_0},\qquad n\ge0.\]
Then
\begin{enumerate}
\item $C_0\sbs C_1\sbs\dots$ and $\de(C_n)\sbs\sum_{i+j=n}C_i\ts C_j$.
\item $\C_n:=C_n\cap\C=\Ker(\bde_{n+1})$.
\item $\C_{n+1}=\bde\inv(\C_n\ts\C_n)$.
\end{enumerate}
\end{lemma}
\begin{proof}
\clm1
This follows from
\cite[9.0.0.i]{sweedler_hopf} and \cite[9.1.6]{sweedler_hopf}.

\clm2
For $n=0$ we have $\bar C_0=0=\Ker(\bde_1)$.
If $x\in\Ker(\bde_{n+2})$, then $\bde(x)\in
(\C_{n}\ts \C)\cap(\C\ts \C_{n})=\C_n\ts \C_{n}$ by induction.
Therefore $\de(x)\in C_n\ts C_n+C_0\ts C+C\ts C_0\sbs C_n\ts C+C\ts C_0$, hence $x\in \C_{n+1}$.
Conversely, if $x\in\C_{n+1}$, then $\bde(x)\in\C_n\ts \C=\Ker(\bde_{n+1})\ts \C$ by induction.
Therefore $\bde_{n+2}(x)=0$.

\clm3
We have seen that if $x\in \C_{n+1}=\Ker(\bde_{n+2})$, then $\bde(x)\in\C_n\ts\C_n$.
Conversely, if $x\in\C$ and $\bde(x)\in\C_n\ts\C_n\sbs \C_n\ts\C$, then
$\de(x)\in C_n\ts C+C\ts C_0$, hence $x\in \C_{n+1}$.
\end{proof}

\begin{definition}
A coalgebra $(C,\de,\eps)$ is called
\begin{enumerate}
\item
\idef{simple} if $C\ne0$ and the only subcoalgebras of $C$ are $0$ and $C$.
%\item
%\idef{connected} if $C$ contains a unique simple coalgebra,
%and it is 1-dimensional.
\item
\idef{connected}
if the coradical $\corad(C)$ (the sum of simple subcoalgebras)
is 1-dimensional.
\item \idef{conilpotent} if it is coaugmented
and $\bar C=\bigcup_{n\ge1}\Ker(\bar\de_n)$.
%for all $x\in \bar C$
%we have $\bar\de_n(x)=0$ for some $n\ge1$.
\end{enumerate}
\end{definition}

\begin{remark}
A coalgebra $C$ is connected \ifft it is pointed (all simple subcoalgebras of $C$ are 1-dimensional) and irreducible ($C$ has a unique simple subcoalgebra).
Some authors use conilpotency as the definition of connectedness for coalgebras (see \eg \cite[\S B.3]{quillen_rational}).
\end{remark}

\begin{theorem}
For a coalgebra $C$ \tfae
\begin{enumerate}
\item $C$ is conilpotent \wrt some coaugmentation.
\item
There exists an increasing filtration $C_0\sbs C_1\sbs\dots\sbs C$ by subspaces such that
$\de(C_n)\sbs\sum_{i+j=n}C_i\ts C_j$,
$C=\bigcup_{n\ge0}C_n$ and
$\dim C_0=1$.

\item $C$ is connected.
\end{enumerate}
\end{theorem}

\begin{proof}
\dir12
Let $C_0=K1_C$ and $C_{n+1}=C_n\wedge C_0$ for $n\ge0$.
%=\sets{x\in C}{\de(x)\in C_n\ts C+C\ts C_0}$.
We have seen that $C_0\sbs C_1\sbs\dots\sbs C$,
that $\de(C_n)\sbs\sum_{i+j=n}C_i\ts C_j$
and that $C_n\cap\C=\Ker(\bde_{n+1})$.
This implies that $\bigcup_n C_n=C$.

\dir21
Consider the coaugmentation corresponding to the group element
$1_C\in C_0$.
Let $\C_n=C_n\cap\bar C$ for $n\ge0$.
Then $\bde(\C_n)\sbs\sum_{i=1}^{n-1}\C_i\ts \C_{n-i}$.
Therefore $\bde_{n+1}(\C_n)=0$,
hence $\C_n\sbs\Ker(\bde_{n+1})$ and $\bigcup_n\Ker(\bde_n)=\C$.

\dir23
The coalgebra $C_0$ contains all simple subcoalgebras of $C$ by
\cite[11.1.1]{sweedler_hopf}.
Since $\dim C_0=1$, we have $\corad(C)=C_0$.

\dir32
Let $C_0=\corad(C)$.
The coradical filtration of $C$, defined by
$C_{n+1}=C_n\wedge C_0$,
satisfies $C_0\sbs C_1\sbs\dots$,
%by \cite[9.0.0.i]{sweedler_hopf},
$\de(C_n)\sbs\sum_{i+j=n}C_i\ts C_j$,
%by \cite[9.1.6]{sweedler_hopf}.
and $C=\bigcup_{n\ge0}C_n$
by \cite[9.0.4]{sweedler_hopf}.
By assumption, $\dim C_0=1$.
\end{proof}

\begin{example}
For a vector space $V$ over a field $K$,
the cofree coalgebra $\Tc(V)=\bop_{n\ge0}V^{\ts n}$
is equipped with the coproduct and the counit \eqref{coprod1}
\begin{equation}
\de(v_1\dots v_n)=\sum_{i=0}^n v_1\dots v_i\ts v_{i+1}\dots v_n
,\qquad \eps:\Tc(V)\to V^{\ts 0}=K.
\end{equation}
It is a conilpotent coalgebra with the coaugmentation $\eta:K=V^{\ts 0}\to \Tc(V)$.
We obtain a functor $\Tc:\Vect\to\CC$
from the category of vector spaces over $K$
to the category of connected coalgebras.
\end{example}

\begin{theorem}[{\sw[12.0.2]}]
\label{free univ}
The functor $\Tc:\Vect\to\CC$ is right adjoint to the forgetful functor $F:\CC\to\Vect$, $C\mto\bar C$.
Explicitly, there is a canonical isomorphism
\[\Hom_{\Vect}(\bar C,V)\iso\Hom_{\Coalg}(C,\Tc V),\qquad
C\in\CC,\, V\in\Vect,
\]
where $f:\bar C\to V$ corresponds to $\tl f:C\to \Tc V$
defined by $\tl f(1)=1$ and $\tl f|_{\bar C}=\sum_{n\ge1}f^{\ts n}\bar\de_n$.
\end{theorem}

\subsection{Universal property of quasi-shuffle algebras}
For a non-unital algebra $(A,\circ)$, consider the composition
\[
\Tc(A)\ts \Tc(A)\to
(A\ts K)\oplus (K\ts A)\oplus (A\ts A)\to A,\]
where the first map is the projection and the second map is given by
\[
a\ts 1_K\mto a,\qquad
1_K\ts b\mto b,\qquad
a\ts b\mto a\circ b,\qquad
a,b\in A.\]
Since the coalgebra $\Tc(A)\ts \Tc(A)$ is conilpotent,
the above map induces a coalgebra morphism
$*:\Tc(A)\ts \Tc(A)\to\Tc(A)$
by \cref{free univ}.
It is precisely the quasi-shuffle product
\cite{newman_cofree,loday_algebra,aguiar_monoidal}.
This endows $\Tc(A)$ with the structure of a conilpotent bialgebra.

By \cite[9.2.2]{sweedler_hopf} every conilpotent bialgebra $(H,\mu,\eta,\de,\eps)$ is a Hopf algebra.
Explicitly, the antipode is given by $S(1)=1$ and
%(\cf \cite{loday_structure})
\[S(x)=\sum_{n\ge1}(-1)^n\mu_n\bar \de_n(x),\qquad x\in\bar H=\Ker\eps,\]
where $\mu_n:H^{\ts n}\to H$ is defined by $\mu_1=\id$ and $\mu_n=\mu(\id\ts\mu_{n-1})$ for $n\ge2$.

\begin{theorem}[{\cite{newman_cofree}}]
For every non-unital algebra $(A,\circ)$,
the product $*$ equips $\QS(A)=\Tc(A)$ with the conilpotent Hopf algebra structure,
called the \idef{quasi-shuffle} algebra of $A$.
The functor $\QS:\Ass_\nun\to\Hopf_\con$ is right adjoint to the forgetful functor
$F:\Hopf_\con\to\Ass_\nun$, $H\mto\bar H$.
Equivalently, for every $A\in\Ass_\nun$ and $H\in\Hopf_{\con}$,
there is a canonical isomorphism
\[\Hom_{\Ass_\nun}(\bar H,A)\iso\Hom_{\Hopf_\con}(H,\QS (A)).\]
\end{theorem}

\begin{remark}
Let us show that for an algebra morphism $f:\bar H\to A$,
the induced map $\tl f:H\to\QS(A)$ from Theorem \ref{free univ}
is a morphism of Hopf algebras.
It is enough to show for $a,b\in\bar H$ that
$\tl f(a)*\tl f(b)=\tl f(ab)$.
We have
\[\tl f(a)*\tl f(b)
=\sum_{\ov{\si\in\qsh(i,j;m)}{i,j,m\ge1}}
\si_* f^{\ts(i+j)}(\bar\de_i(a)\ts\bar\de_j(b))
=\sum_{\ov{\si\in\qsh(i,j;m)}{i,j,m\ge1}}
f^{\ts m}\si_* (\bar\de_i(a)\ts\bar\de_j(b))
\]
Applying Lemma \ref{qs coprod},
we obtain
$\tl f(a)*\tl f(b)=
\sum_{m\ge1} f^{\ts m}\bar\de_m(ab)=\tl f(ab)$.
\end{remark}

\def\A{A}
\subsection{Relation to CTD algebras}
\label{CTD}
A \idef{CTD algebra}
(commutative tridendriform algebra)
\cite{loday_algebra}
is a triple $(A,\circ,\prec)$, where $\circ,\prec$ are bilinear operations on a vector space $A$ such that
\begin{enumerate}
\item $(A,\circ)$ is commutative and associative.
\item $(a\circ b)\prec c=a\circ(b\prec c)$.
\item $(a\prec b)\prec c=a\prec(b\prec c+c\prec b+b\circ c)$.
\end{enumerate}
It is called unital if $(A,\circ)$ is unital.
The product
\begin{equation}
a*b=a\prec b+b\prec a+a\circ b
\end{equation}
is also commutative and associative.
Let $\CTD$ denote the category of CTD algebras.

\begin{example}\label{ex:CTD1}
A \idef{Rota-Baxter algebra} is a commutative non-unital algebra $(A,\circ)$ equipped with a linear map $P:A\to A$ satisfying
\begin{equation}\label{RB eq}
P(a)\circ P(b)=P(a\circ P(b)+P(a)\circ b+a\circ b).
\end{equation}
The map $P$ is called a \idef{Rota-Baxter operator}.
Such an algebra has a CTD algebra structure with $a\prec b=a\circ P(b)$.
Equation \eqref{RB eq} implies that the map $P:(A,*)\to (A,\circ)$ is an algebra morphism.
Conversely, if $(A,\circ,\prec)$ is a unital CTD algebra (meaning that $(A,\circ)$ is a unital algebra), then the map $P:A\to A$, $a\mto 1\prec a$, is a Rota-Baxter operator.
Indeed, we have $a\circ P(b)=a\prec b$, hence
\[P(a)\circ P(b)=(1\prec a)\prec b
=1\prec(a*b)
=P(a\circ P(b)+P(a)\circ b+a\circ b)
.\]
This implies that unital CTD algebras can be identified with unital Rota-Baxter algebras.
\end{example}

\begin{example}
Let $(A,\circ)$ be a commutative non-unital algebra.
Define bilinear operations
$*$ on $\QS(\A)=\bop_{n\ge0}A^{\ts n}$
and $\circ,\prec$ on $\QSp(\A)=\bop_{n\ge1}A^{\ts n}$
inductively by the rules $1_K*u=u*1_K=u$ and
\begin{enumerate}
\item $au\circ bv=(a\circ b)(u*v)$ for $a,b\in \A$ and $u,v\in\QS(\A)$.
\item $au\prec bv=a(u*bv)$ for $a,b\in \A$ and $u,v\in\QS(\A)$.
\item $u*v=u\prec v+v\prec u+u\circ v$
for $u,v\in\QSp(\A)$.
\end{enumerate}
In this way $\QSp(A)$ is equipped with a CTD algebra structure.
Note that the product $*$ is commutative and we have
\[(au)*(bv)=a(u*bv)+b(au*v)+(a\circ b)(u*v),\]
hence $*$ is the quasi-shuffle product.
If $A$ is unital, then the CTD algebra $\QSp(A)$ is unital, with the identity element $1_A\in A\sbs\QSp(A)$.
Therefore $\QSp(A)$ is equipped with the Rota-Baxter operator $P(u)=1_A\prec u=1_A\ts u$.
\end{example}

\begin{theorem}[{\cite{loday_algebra}}]
The functor $\QSp:\Com_{\nun}\to\CTD$
is left adjoint to the functor
\[F:\CTD\to\Com_\nun,\qquad (A,\circ,\prec)\mto(A,\circ).\]
Explicitly, for every $A\in\Com_\nun$ and $C\in\CTD$,
there is a canonical isomorphism
\[\Hom_{\CTD}(\QSp(A),C)\iso\Hom_{\Com_\nun}(A,C),\]
where a morphism $f:A\to C$ corresponds to
$\bar f:\QSp(A)\to C$ such that
$\bar f(au)=f(a)\prec \bar f(u)$ and $\bar f(a)=f(a)$
for $a\in A$ and $u\in\QSp(A)$.
\end{theorem}

\begin{remark}
For a unital commutative algebra $A$, the CTD algebra $\ub\QS(A)=\bop_{n\ge1}A^{\ts n}$ is also a unital  Rota--Baxter algebra.
It is the free unital Rota--Baxter algebra generated by the algebra $A$ (\cf \cite{guo_baxter}).
For a non-unital commutative algebra $A$, we consider the unital algebra $A^+=K\oplus A$.
Then the free (non-unital) Rota--Baxter algebra $F_{\RB}(A)$
generated by the algebra $A$ is the Rota--Baxter subalgebra of $\ub\QS(A^+)$ generated by $A$.
It is equal to $\bop_{n\ge0}(A^+)^{\ts n}\ts A$.
This construction first appeared in \cite{cartier_structure}
for the free non-unital commutative algebra $A=\bop_{n\ge1}S^n(KX)$
generated by a set~$X$, and in \cite{guo_internal} for arbitrary $A$.
\end{remark}

\section{Multiple zeta values in \tpdf{\la}{lambda}-rings}

\subsection{Truncated MZV}
Let $K$ be a field of characteristic zero.
It has a \la-ring structure with Adams operations $\psi^n=\id$ for $n\ge1$.
Let $A$ be a $K$-algebra and a \la-ring such that $\psi^n|_K=\id$ for $n\ge1$.
% (see Appendix \ref{lambda}).
Let $\QS(A)=\bop_{n\ge0}A^{\ts n}$ be the quasi-shuffle algebra
of $A$ and let $\ep=1_K$ denote the identity element of $\QS(A)$.
For $n\ge1$, we define the \idef{truncated \mzf}
\begin{equation}
Z_n:\QS(A)\to A,\qquad
a_1\ts\dots\ts a_l\mto \sum_{n>n_1>\dots>n_l>0}\prod_{i=1}^l \psi^{n_i}(a_i),\qquad l\ge1,
\end{equation}
and $Z_n(\ep)=1$.
For $n\ge2$, we have
\begin{equation}
Z_n(au)=\sum_{n>m>0}\psi^m(a) Z_m(u),\qquad
a\in A,\,u\in \QS(A),
\end{equation}
where $au=a\ts u$.

\begin{lemma}\label{lm:homom1}
The map $Z_n:\QS(A)\to A$ is an algebra morphism for all $n\ge1$.
\end{lemma}
\begin{proof}
The map $Z_1$ is the composition $\QS(A)\to K\emb A$, hence an algebra morphism.
By induction on $n$, we have, for $a,b\in A$ and $u,v\in \QS(A)$,
\begin{multline*}
Z_n(au) Z_n(bv)
=\sum_{n>m,m'>0}\psi^m(a)\psi^{m'}(b)Z_m(u)Z_{m'}(v)
\\
=\sum_{n>m>0}\rbig{
\psi^m (a) Z_m(u)Z_m(bv)
+\psi^m(b)Z_m(au)Z_m(v)
+\psi^m(a\circ b)Z_m(u)Z_m(v)}\\
=Z_n\rbig{a(u*bv)+b(au*v)+(a\circ b)(u*v)}
=Z_n(au*bv).
\tag*{\qedhere}
\end{multline*}
\end{proof}

\begin{remark}
The Adams operations $\psi^n:A\to A$ are algebra morphisms satisfying $\psi^m\psi^n=\psi^{mn}$.
Therefore they induce algebra morphisms $\psi^n:\QS(A)\to\QS(A)$
satisfying $\psi^m\psi^n=\psi^{mn}$.
This implies that $\QS(A)$ is itself a \la-ring.
The algebra morphism $Z_n:\QS(A)\to A$ is a morphism of $\la$-rings.
\end{remark}

Consider the \idef{multiple polylogarithm}
$\Lt:\QS(A)\to A\pser t$,
defined by
\begin{equation}
\Lt(a_1\ts\dots\ts a_l)
=\sum_{n_1>\dots>n_l>0}t^{n_1}\prod_{i=1}^l \psi^{n_i}(a_i)
\in A\pser t,\qquad l\ge1,
\end{equation}
and $\Lt(\ep)=1$.
For $a\in A$, consider the continuous $A$-linear operators
\begin{equation}\label{PY}
\P_a:tA\pser t\to tA\pser t,\quad t^n\mto\psi^n(a)t^n,\qquad
\cT:A\pser t\to tA\pser t,\quad f\mto \frac t{1-t}f.
\end{equation}
Note that the map $A\to\End_A(tA\pser t)$, $a\mto \P_a$,
is a ring morphism.

\begin{theorem}
\label{Zt-PY}
For $a\in A$ and $u=a_1\ts\dots\ts a_l\in A^{\ts l}$, we have
\[\Lt(au)=\P_a \cT\Lt(u)\qquad
\Lt(u)=\P_{a_1}\cT\dots \P_{a_l}\cT(1).\]
\end{theorem}
\begin{proof}
For $u=a_1\ts\dots\ts a_l$, we have
\[\bLt(u)
:=\sum_{n\ge1}Z_n(u)t^n
=\sum_{n>n_1>\dots>n_l>0}t^{n}\prod_{i=1}^l \psi^{n_i}(a_i)
=\frac t{1-t}\Lt(u)=\cT L_t(u).\]
%hence $\bLt=\cT\Lt$.
Therefore
\[
\Lt(au)=\sum_{n\ge1}\psi^n(a)Z_n(u)t^n
=\P_a\bLt(u)=\P_a\cT \Lt(u).
\]
%We have
%\[\bLt(au)
%=\sum_{n>m>0}\psi^m(a)Z_m(u)t^n
%=\frac t{1-t}\sum_{m>0}\psi^m(a)Z_m(u)t^m=\cT \P_a \bLt(u).
%\]
%hence
%$\Lt(au)=\P_a\cT\Lt(u)$.
For the second equation we note that $\Lt(\ep)=1$.
\end{proof}

\begin{corollary}\label{cor:Zt}
For $a\in A$ and $k\in\bN^{l}$,
we have
\[\Lt(a^{k_1}\ts\dots\ts a^{k_l})
=\P_a^{k_1}\cT\dots \P_a^{k_l}\cT(1).\]
\end{corollary}
\begin{proof}
We note that
$\P_{a^n}=\P_a^n$.
\end{proof}

\begin{remark}\label{iter}
The above result may be compared to the iterated integral formula
\cite{zagier_values,hoffman_algebra}
for the multiple zeta values and the multiple polylogarithms.
As in the introduction, let $U=\bop_{i\ge1}\bQ z_i$ with $z_i\circ z_j=z_{i+j}$.
We can identify $U$ with $a\bQ[a]$ via $z_i\mto a^i$.
Let
\[L_t:\QS(U)\to\bC\pser t,\qquad L_t(z_{k_1}\dots z_{k_l})
=\sum_{n_1>\dots>n_l>0}t^{n_1}\prod_i\frac1{n_i^{k_i}}.
\]
Consider the operators $X:t\bC\pser t\to t\bC\pser t$
and $Y:\bC\pser t\to t\bC\pser t$ given by
\[(Xf)(t)=\int_0^t f(s)\frac{ds}{s},\qquad
(Yf)(t)=\int_0^t f(s)\frac{ds}{1-s},
\]
satisfying $Y=X\cT$ and $X(t^n)=\frac1n t^n$,
where $\cT(f)=\frac t{1-t}f$.
Then \cite{hoffman_algebra}
%\cite[Theorem 6.1]{hoffman_algebra}
\begin{equation}
L_t(z_ku)=X^{k-1}Y L_t(u)=X^k\cT L_t(u).
\end{equation}
The iterated integral formula can be written in the form
\begin{equation}
L_t(z_{k_1}\dots z_{k_l})
=X^{k_1}\cT\dots X^{k_l}\cT(1).
\end{equation}
See also \cite{singer_q} and \eqref{Singer}
for a similar result for multiple $q$-zeta values.
\end{remark}

\subsection{Untruncated MZV}
Assume now that $A$ is a complete \la-ring \cite[\S2]{getzler_mixed},
equipped with a decreasing filtration
\[A=F^0A\sps A^+=F^1A\sps F^2A\sps\dots\]
such that $F^iA\cdot F^j A\sbs F^{i+j}A$ and $\psi^n(F^iA)\sbs F^{ni}A$.
We define the \idef{\mzf}
\begin{equation}
Z:\QS(A^+)\to A,\qquad u=a_1\ts\dots\ts a_l
\mto\sum_{n_1>\dots>n_l>0}\prod_{i=1}^l \psi^{n_i}(a_i).
\end{equation}
Equivalently,
\begin{equation}
Z(u)=\lim_{n\to\infty}Z_n(u)=\Lt(u)|_{t=1}.
\end{equation}
Here evaluation at $t=1$ is understood with respect to the
filtration topology on $A$.

\begin{theorem}
The map $Z:\QS(A^+)\to A$ is an algebra morphism.
Moreover,
\[Z(a_1\ts\dots\ts a_l)=\P_{a_1}\cT\dots \P_{a_l}\cT(1)|_{t=1}.\]
\end{theorem}
\begin{proof}
This follows from Lemma \ref{lm:homom1}
and Theorem \ref{Zt-PY}.
\end{proof}

For a (non-unital) subalgebra $B\sbs A$,
consider the restriction
\begin{equation}
Z:\QS(B^+)\to A,\qquad
B^+=B\cap A^+.
\end{equation}
We denote its image by $\cZ(B^+)$ and call it the \idef{algebra of \mzv} on $B^+$.

\begin{remark}
More generally, given vector spaces $U\sbs V$, let
\[\T(U,V)=K\oplus (U\ts\T(V))\sbs\T(V)
=\bop\nolimits_{n\ge0}V^{\ts n}.\]
If $B\sbs A$ is a (non-unital) subalgebra, then the subspace
$\T(B^+,B)\sbs\QS(A)$ is a subalgebra, denoted by $\QS(B^+,B)$.
The \idef{\mzf}
\[Z:\QS(B^+,B)\to A,\qquad
a_1\ts\dots\ts a_l
\mto\sum_{n_1>\dots>n_l>0}\prod_{i=1}^l \psi^{n_i}(a_i),\]
is again an algebra morphism.
We denote its image by $\cZ(B)$.
More generally, for every subspace $V\sbs A$,
let $\cZ(V)$ be the image of the linear map $Z:\T(V^+,V)\to A$,
where $V^+=V\cap A^+$.
%If $B$ is closed under Adams operations
%(a possibly non-unital \la-subring of $A$),
%then $\QS(B^+,B)$ is a \la-ring and
%$Z:\QS(B^+,B)\to A$ is a \la-ring morphism.
%Therefore $\cZ(B)\sbs A$ is a \la-subring.
\end{remark}

\begin{example}
Given a commutative graded algebra $A=\bop_{k\in\bZ}A_k$ over $\bC$, we equip it with the \la-ring structure having Adams operations
%(\cf Example \ref{graded lambda})
\[\psi^n(a)=n^{-k} a,\qquad a\in A_k,\, n\ge1.\]
The corresponding multiple zeta values converge on $\QS(A_{\ge2},A_{\ge1})$, where $A_{\ge m}=\bop_{k\ge m}A_k$.

In particular, let $A=\bC[a]$ be graded with $\deg a=1$.
Then
\[Z(a^{s_1}\ts\dots\ts a^{s_l})
=\sum_{n_1>\dots>n_l>0}\prod_{i=1}^l n_i^{-s_i}a^{s_i}
=\ze(s_1,\dots,s_l)a^{\sum_i s_i}
\]
for $s\in\bZ_{\ge1}^l$ with $s_1\ge2$.
Consider the $\bQ$-subalgebra $U=a\bQ[a]\sbs \bC[a]$
and the multiple zeta function $Z:\QS(aU,U)\to \bC[a]$.
Its image $\cZ(U)\sbs \bC[a]$ is a graded algebra equal to
$\bop_{k\ge0}\cZ_k a^k\iso \bop_{k\ge0}\cZ_k$,
where $\cZ_k\sbe\bC$ is the space spanned by weight $k$ multiple zeta values (corresponding to collections $(s_i)_i$ satisfying $\sum_i s_i=k$).
The classical algebra $\cZ\sbs \bC$ of \MZV is the image of the evaluation map $\cZ(U)\emb\bC[a]\to\bC$, $f\mto f(1)$.
It is conjectured in \cite{hoffman_algebra} that the map $\cZ(U)\to\cZ$ is an isomorphism.
\end{example}

\subsubsection{Weights}
Let $A$ be a complete \la-ring as before and let $B\sbs A$ be a subalgebra equipped with an increasing (weight) filtration $(W_iB)_{i\ge0}$ such that
$W_0B=0$ and $W_iB\circ W_jB\sbs W_{i+j}B$.

\begin{remark}
Given vector spaces $U,V$ with increasing filtrations
$(W_iU)_{i}$ and $(W_iV)_{i}$, we equip $U\ts V$ with the filtration
$W_n(U\ts V)=\sum_{i+j=n}W_iU\ts W_jV$.
If $U$ is equipped with a filtration $(W_iU)_i$ and $f:U\to V$ is a surjective linear map,
then $V$ can be equipped with the filtration $W_iV=f(W_iU)$.
A subspace $V\sbs U$ can be equipped with the filtration $W_iV=V\cap W_iU$.
%One defines $\gr^W_iU=W_iU/W_{i-1}U$ for $i\in\bZ$.
\end{remark}

The algebra $Q=\QS(B^+,B)$ inherits the weight filtration such that $W_0Q=K=B^{\ts0}$.
% (since $W_0B=0$).
Therefore the image $\cZ(B)$ of $Z:Q\to A$ also inherits the weight filtration.
Assuming that $\dim W_iB<\infty$ for all $i\ge0$, we conclude that the same is true for $Q$ and $\cZ(B)$ (recall that $W_0B=0$).
Define the Poincar\'e series of $\cZ(B)$
\begin{equation}
P(\cZ(B),t)=\sum_{i\ge0}\dim \gr^W_i \cZ(B) t^i,\qquad
\gr^W_i\cZ(B)=W_i\cZ(B)/W_{i-1}\cZ(B),
\end{equation}
with the convention $W_{-1}=0$.

\section{Multiple \tpdf{$q$}{q}-zeta values}
Let $K$ be a field of characteristic zero
and let $A=K\pser q$ be the \la-ring with Adams operations $\psi^n(f)=f(q^n)$.
It is a complete \la-ring \wrt the decreasing filtration $(F^iA)_{i\ge0}$ defined by $F^iA=q^iA$ for $i\ge0$.
Let $A^+=F^1A=qK\pser q$.
Given a subalgebra $B\sbs A$, let $B^+=B\cap A^+$.
By the previous results, the \mzf
\[Z:\QS(B^+,B)\to K\pser q,\qquad
f_1\ts\dots\ts f_k\mto
\sum_{n_1>\dots>n_k>0}\prod_{i=1}^k f_i(q^{n_i}),
\]
is an algebra homomorphism.
Let $\qZ(B)\sbs K\pser q$ denote its image.
First, we review the known results and conjectures concerning
these algebras
and then investigate a particular algebra of this type.

\subsection{Multiple \tpdf{$q$}{q}-zeta values}
\label{sec:qMZV}
Consider the algebras
\begin{equation}
\R=K\sbr{q,\tfrac 1{q-1}},\qquad
\R_0=K\sbr{\tfrac{q}{1-q}}
=\sets{f\in \R}{\deg f\le0},
\end{equation}
where $\deg(f/g)=\deg f-\deg g$ for $f,g\in K[q]\ms\set0$
and $\deg(0)=-\infty$.
These algebras have the following algebro-geometric interpretation.
For $f=\sum_{i\in\bZ}f_iq^i\in K\lser q$,
let
\[\ord(f)=\inf\sets{i\in\bZ}{f_i\ne0}.\]
In particular, $\ord(0)=\infty$.
For $f\in K(q)$ and $c\in K$, let
$\ord_c(f)=\ord f(q+c)$.
It is the zero order of $f$ at $q=c$ for $f\in K[q]$.
We also define $\ord_\infty(f)=\ord f(q\inv)=-\deg f$.
Then
\begin{gather*}
\cO(\bA^1_K\ms\set 1)=K\sbr{q,\tfrac1{q-1}}=\R,\\
\cO(\bP^1_K\ms\set1)
=\sets{f\in \cO(\bA^1_K\ms\set 1)}{\ord_\infty(f)\ge0}
=\sets{f\in \R}{\deg f\le0}=\R_0.
\end{gather*}
If $0\ne f\in\R_0$, then $\ord_c(f)\ge0$ for all $c\ne1$, while $\ord_1(f)\le0$.
More generally, define
\begin{equation}\label{Rd}
\R_d=\sets{f\in \R}{\deg f\le -d},\qquad d\ge0.
\end{equation}

Define the algebras of multiple $q$-zeta values
(\cf \cite{bachmann_dimension})
\begin{gather}
\qzp d=\qZ(\Rp_d)=\Im\rbr{Z:\QS(\Rp_d)\to K\pser q},\\
\qz d=\qZ(\R_d)=\Im\rbr{Z:\QS(\Rp_d,\R_d)\to K\pser q}.
\end{gather}

%\begin{equation}
%\qz d=\qz(\R_d)\sbs K\pser q,\qquad \qzp d=\qz(\Rp_d)\sbs K\pser q.
%\end{equation}

Consider an increasing weight filtration of $\R_0$ defined by
$W_0\R_0=0$ and
\begin{equation}\label{R-wt}
W_k\R_0
=\sets{f\in \R_0}{\ord_1(f)\ge-k}
=\sets{f\in\R_0}{(q-1)^kf\in K[q]}
,\qquad k\ge1.
\end{equation}
We have $\dim W_k\R_0=k+1$ for $k\ge1$.
The subalgebras $\R_d$ and $\Rp_d$
inherit the weight filtration.
% $W_k\R_d=\R_d\cap W_k\R_0$.
The algebras $\qzp d$ and $\qz d$ also inherit the weight filtration.

\begin{remark}\label{Z-compare}
It is proved in \cite{hirose_unified} that
$\qz 0=\qzp 0$ and $\qz 1=\qzp 1$.
These equalities were conjectured earlier in
\cite{bachmann_algebraa,bachmann_dimension}.
Although not stated explicitly, the proof in \cite{hirose_unified} also shows that these
equalities are compatible with the weight filtrations:
the reduction in \cite[Lemma~2.3]{hirose_unified}
and the recursion in \cite[Definition~3.6]{hirose_unified}
do not increase the weight.
\end{remark}

\begin{lemma}\label{basis}
Let $(f_i)_{i\ge0}$ be a sequence in $\R_0$ such that $g_i=(q-1)^if_i\in K[q]$ and $g_i(1)\ne0$.
Then $(f_i)_{i\ge0}$ forms a basis of $\R_0$.
\end{lemma}
\begin{proof}
For $k\ge1$, the elements of the subspace $W_k=\sets{f\in\R_0}{\ord_1 (f)\ge-k}$
are of the form $f/(q-1)^k$, where $f\in K[q]$ has degree $\le k$.
Therefore $\dim W_k=k+1$.
The functions $f_0,\dots,f_k$ are contained in $W_k$ and are linearly independent.
Indeed, if $\sum_{i=0}^k a_if_i=0$,
then $0=(q-1)^k\sum_{i=0}^k a_if_i|_{q=1}=a_k g_k(1)$, hence $a_k=0$.
Similarly, $a_0=\dots=a_k=0$.
We conclude that $f_0,\dots,f_k$ form a basis of~$W_k$.
Therefore $(f_i)_{i\ge0}$ forms a basis of $\R_0$.
\end{proof}

In the theory of multiple $q$-zeta values,
it is common to fix a basis $\bff=(f_s)_{s\in S}$ of an algebra $B\sbs \R_0$ and to define
\begin{equation}\label{Z-model}
%Z^\bff(s_1\dots s_l)=
Z^\bff(s_1,\dots,s_l)=Z(f_{s_1}\ts\dots \ts f_{s_l})\in K\pser q
\end{equation}
for $s\in S^l$ with $f_{s_1}\in B^+$.
In this context, the basis $\bff$ is called a \idef{model}.
See \cite{schlesinger_some,zudilin_diophantine,
bradley_multiple,zhao_multiple,
okounkov_hilbert,bachmann_short,bachmann_algebra}
for various choices of such models.
Assume that $B^+$ has a basis $(f_s)_{s\in S_+}$ for a subset $S_+\sbs S$.
Then the algebra $\qZ(B)$ is spanned by the values $Z^\bff(s_1,\dots,s_l)$ for $s\in S^l$ with $s_1\in S_+$ (or $s=()\in S^0$).
Similarly, the algebra $\qZ(B^+)$ is spanned by the values $Z^\bff(s_1,\dots,s_l)$ for $s\in S_+^l$.

\begin{example}
Let
\begin{equation}
f_{d,k}=\frac{q^{k-d}}{(1-q)^k}\in\R_0,\qquad k\ge d\ge0.
\end{equation}
so that
\begin{equation}
f_{d+1,k}=f_{d,k}+f_{d,k-1},\qquad k\ge d+1.
\end{equation}
Then
\begin{enumerate}
\item $\R_d$ has the basis $f_{d,i}$ for $i\ge d$.
\item $\Rp_d$ has the basis $f_{d,i}$ for $i>d$.
\item $W_k\R_d$ has the basis $f_{d,i}$ for $d\le i\le k$,
where $k\ge1$.
\item $W_k\Rp_d$ has the basis $f_{d,i}$ for $d< i\le k$.
\end{enumerate}
The basis $\bff^\SZ=\rbr{z_k=f_{0,k}}_{k\ge0}$ of $\R_0$ is called the
Schlesinger-Zudilin model
\cite{schlesinger_some,zudilin_diophantine}
and the basis $\bff^\BZ=(u_k=f_{1,k})_{k\ge1}$ of $\R_1$ is called the Bradley-Zhao model
\cite{bradley_multiple,zhao_multiple}.
Note that
\begin{equation}
z_m\circ z_n=z_{m+n},\qquad
u_m\circ u_n=u_{m+n}+u_{m+n-1}.
\end{equation}
By Corollary \ref{cor:Zt}, we have (\cf \cite{singer_q})
\begin{equation}\label{Singer}
\Lt(z_{s_1}\ts\dots\ts z_{s_k})
=\sum_{n_1>\dots>n_k>0}t^{n_1}
\prod_{i=1}^k\rbr{\frac{q^{n_i}}{1-q^{n_i}}}^{s_i}
=X^{s_1}\cT\dots X^{s_k}\cT(1),
\end{equation}
where $X=\P_a$ for $a=z_1=\frac q{1-q}$ \eqref{PY}, meaning that $(Xf)(t)=\sum_{i\ge1}f(q^it)$ for $f\in tA\pser t$.
The operator $X=\P_a$ is actually a Rota-Baxter operator \eqref{RB eq},
since $a=\frac q {1-q}$ satisfies
\begin{equation}
\psi^m(a)\psi^n(a)=\psi^{m+n}(a)(\psi^m(a)+\psi^n(a)+1).
\end{equation}
\end{example}

\begin{example}
Consider the Eisenstein series (for even $k>0$)
\[E_k=1-\frac{2k}{B_k}\sum_{n,d\ge1}n^{k-1}q^{nd},\]
where $B_k\in\bQ$ are Bernoulli numbers.
Using the Euler operator $D=q\frac{d}{dq}$, we obtain
\[\sum_{n,d\ge1}n^{k-1}q^{nd}
=\sum_{d\ge1}\psi^d\rbr{D^{k-1}\frac q{1-q}}
=Z(Q_k),\qquad Q_k=D^{k-1}\frac q{1-q}.
\]
The polynomials $P_k(q)$ defined by
\[\frac{qP_k(q)}{(1-q)^{k+1}}=\sum_{n\ge0}(n+1)^kq^{n+1}
=D^k\frac q{1-q}=Q_{k+1}\]
are the Eulerian polynomials satisfying $P_k\in\bN[q]$ and $\deg P_k=k-1$  for $k\ge1$ (see \eg \cite{hirzebruch_eulerian}).
Therefore $P_k(1)\ne0$ and $\deg Q_{k+1}=-1$ for $k\ge1$.
Applying \cref{basis}, we conclude that $(Q_k)_{k\ge1}$ is a basis of~$\Rp_0$.
Up to some rational scalars, this basis was considered in \cite{bachmann_algebra}.
\end{example}

It is conjectured in \cite{bachmann_algebraa,bachmann_dimension}
that $\qzp 0=\qZ(W_3\Rp_0)$, where
$W_3\Rp_0=\angs{f_{0,i}}{1\le i\le 3}$.
Moreover,
\begin{equation}\label{BK_conj}
\sum_{i\ge0}\dim\gr^W_i(\qzp 0) t^i
=\frac1{1-t-t^2-t^3+t^6+t^7+t^8+t^9}.
\end{equation}
Similarly, it is conjectured in \cite{okounkov_hilbert}
that $\qzp {1}=\qZ(W_5 \Rp_1)$, where
$W_5\Rp_1=\angs{f_{1,i}}{2\le i\le 5}$.
Moreover
\begin{equation}
\sum_{i\ge0}\dim\gr^W_i(\qzp 1) t^i
=\frac1{1-t^2-t^3-t^4-t^5+t^8+t^9+t^{10}+t^{11}+t^{12}}.
\end{equation}
Generally, our computations suggest that
$\qzp d=\qZ(W_n\Rp_d)$ for some $n=n(d)$.
For example,
\begin{equation}
\qzp {2}=\qZ(W_9 \Rp_2),\qquad
\qzp {3}=\qZ(W_{13} \Rp_3),\qquad
\qzp {4}=\qZ(W_{17} \Rp_4).
\end{equation}

\subsection{Duality}
\label{sec:duality}
Recall that the algebra $\qz 0=\qZ(\R_0)\sbs A=K\pser q$ is the image of the algebra homomorphism $Z:\QS(\Rp_0,\R_0)\to A$.
In this section we will describe the conjectural kernel of this map.

The algebra $\R_0$ has the basis $(z_i)_{i\ge0}$ and the algebra $\Rp_0$ has the basis $(z_i)_{i\ge1}$,
where $z_i=\frac{q^i}{(1-q)^i}$ has weight $\max\set{i,1}$ for $i\ge0$ \eqref{R-wt}.
We identify $z_i$ with the elements $x^iz_0$
of the free algebra $\eH=K\ang{x,\y}$.
% and the elements $z_i=x^i\y\in\eH$ for $i\ge0$.
Then
\[\eH^1:=K\oplus \eH \y=K\ang{z_0,z_1,\dots}\iso\QS(\R_0),\]
\[\eH^0:=K\oplus x\eH \y
=K\oplus\bop_{i\ge1}z_i\eH^1\iso\QS(\R_0^+,\R_0).\]
Under this identification, we have
\[Z(z_{s_1}\dots z_{s_k})
=\sum_{n_1>\dots>n_k>0}\prod_i\rbr{\frac{q^{n_i}}{1-q^{n_i}}}^{s_i}.\]

Let us equip $\eH$ and $\eH^0$ with the anti-involution
\[\ta:\eH\to\eH,\qquad \ta(uv)=\ta(v)\ta(u),\qquad \ta(x)=\y,\qquad \ta(\y)=x.\]

\begin{theorem}
[See {\cite{takeyama_algebra,zhao_uniform,ebrahimifard_duality}}]
Under the identification $\eH^0\iso\QS(\Rp_0,\R_0)$,
we have
\[Z(\ta u)=Z(u),\qquad u\in\eH^0.\]
\end{theorem}
\begin{proof}
Let $a=\frac q{1-q}$ and $X=\P_a$ \eqref{PY}.
We have $Xf=\sum_{k\ge1}\P_{q^k}f=\sum_{k\ge1}f(q^kt)$ for $f\in tA\pser t$.
Let $u=z_{s_1}\dots z_{s_k}=x^{s_1}\y\dots x^{s_k}\y$
so that $\ta u=x\y^{s_k}\dots x\y^{s_1}$.
By Theorem~\ref{Zt-PY}
\[\Lt(u)=X^{s_1}\cT\dots X^{s_k}\cT(1).\]
Therefore
\[\Lt(\ta u)=X\cT^{s_k}\dots X\cT^{s_1}(1)=\sum_{n_1>\dots>n_k>0}\prod_i
\rbr{\frac{q^{n_i}t}{1-q^{n_i}t}}^{s_i}.\]
This implies that $Z(\ta u)=\sum_{n_1>\dots>n_k>0}
\prod_i\rbr{\frac{q^{n_i}}{1-q^{n_i}}}^{s_i}=Z(u)$.
\end{proof}

Let $\eH^0_*$ denote the algebra $\eH^0$ equipped with the quasi-shuffle product $*$ inherited from $\QS(\Rp_0,\R_0)$.
Let $(\ta u-u\col u\in\eH^0)\sbs\eH^0_*$ be the ideal generated by the elements of the form $\ta u-u$.
Define the algebra of \idef{formal multiple $q$-zeta values}
\begin{equation}\label{formal Z}
\Zf=\eH^0_*/(\ta u-u\col u\in\eH^0).
\end{equation}

\begin{conjecture}[See {\cite{takeyama_algebra}}]
\label{conj2}
The canonical surjective map $Z:\Zf\to\qz0$ is an isomorphism.
\end{conjecture}

The algebras $\eH^0_*$ and $\Zf$ inherit the weight filtration from $\QS(\Rp_0,\R_0)$.
Explicitly,
\begin{equation}
\wt(z_{s_1}\dots z_{s_k})=\sum_i \max\set{s_i,1}.
\end{equation}
Note that $\wt(\ta u)=\wt(u)$ for every word $u$ in $\eH^0$.
Combining Conjecture \ref{conj2} and the conjectural formula \eqref{BK_conj}, we obtain a
conjecture about the Poincar\'e series of the algebra $\Zf$
\begin{equation}
P(\Zf,t)
=\sum_{i\ge0}\dim\gr^W_i(\Zf) t^i
=\frac1{1-t-t^2-t^3+t^6+t^7+t^8+t^9}
\end{equation}

\def\B{B}
\def\A{K\pser q}
\subsection{Algebra of regular qMZV}
In this section we will study multiple zeta values
associated with the algebra
$\B=\cO(\bP_K^1\ms\set\infty)=K[q]$
in place of the algebra $\R_0=\cO(\bP_K^1\ms\set1)$ considered earlier.
As before, we consider the algebra homomorphisms
\[
Z:\QS(\B^+)\to \A,\qquad
Z:\QS(\B^+,\B)\to \A,\]
and denote their images by $\Zrp$ and $\Zr$, respectively.
Similarly to \eqref{R-wt}, we equip $\B$ with the weight filtration
given by $W_0\B=0$ and
\[W_k\B=\sets{f\in \B}{\ord_\infty(f)\ge-k}=\sets{f\in \B}{\deg f\le k},\qquad k\ge1,\]
and consider the induced filtrations of $\Zrp$ and $\Zr$.

Consider the basis $\bff=(q^i)_{i\ge0}$ of $\B$ and,
following \eqref{Z-model}, let
\[
Z^\bff(s_1,\dots,s_k)=Z(q^{s_1}\ts\dots\ts q^{s_k}),\qquad
\Lt^\bff(s_1,\dots,s_k)=\Lt(q^{s_1}\ts\dots\ts q^{s_k}),\]
for $s\in\bN^k$ with $s_1\ge1$.
The weight of $Z^\bff(s_1,\dots,s_k)$ is $\le\sum_i \max\set{s_i,1}$.

\begin{lemma}\label{regZ formula}
We have
\[\Lt^\bff(s_1,\dots,s_k)
=\prod_{i=1}^k \frac{q^{s_1+\dots+s_i}t}{1-q^{s_1+\dots+s_i}t},\qquad
Z^\bff(s_1,\dots,s_k)=\prod_{i=1}^k \frac{q^{s_1+\dots+s_i}}{1-q^{s_1+\dots+s_i}}\]
\end{lemma}
\begin{proof}
By Corollary \ref{cor:Zt}, we have
$\Lt^\bff(s_1,\dots,s_k)=\P_q^{s_1}\cT\dots \P_q^{s_k}\cT(1)$.
We have $\P_q(f)=f(qt)$,
since $\P_q(t^n)=\psi^n(q)t^n=(qt)^n$.
Therefore $\P_q^k\cT f=\frac{q^kt}{1-q^kt}f(q^kt)$
and
\[\P_q^{s_1}\cT\dots \P_q^{s_k}\cT(1)
=\prod_{i=1}^k \frac{q^{s_1+\dots+s_i}t}{1-q^{s_1+\dots+s_i}t}.
\qedhere
\]
\end{proof}

\begin{corollary}
The algebras $\Zrp$ and $\Zr$ coincide,
and are generated by the elements
\[f_k=\frac{q^k}{1-q^k},\qquad k\ge1.\]
\end{corollary}

\begin{example}
The algebras $\Zrp$ and $\Zr$ have different weight filtrations,
in contrast to the algebras $\qzp0$ and $\qz0$ (see \cref{Z-compare}).
For example, the element $f_2^2=Z^\bff(2,0)$ has weight $\le2+1=3$,
hence $f_2^2\in W_3\Zr$.
On the other hand,
\[W_3\Zrp
=\spn\set{1,f_1,f_2,f_3,f_1f_2,f_1f_3,f_2f_3,f_1f_2f_3}
\]
and all elements of this spanning set have at most a simple pole at $q=-1$.
Therefore $f_2^2\notin W_3\Zrp$.
We will determine below the Poincar\'e series of $\Zrp$ and $\Zr$
for the two weight filtrations.
\end{example}

For $\bn=(n_1,\dots,n_k)$ with $1\le n_1\le\dots\le n_k$, define
\begin{equation}
f_\bn=\prod_{i=1}^k f_{n_i},\qquad
w(\bn)=n_k+\#\setc{2\le i\le k}{n_i=n_{i-1}}.
\end{equation}
For the empty sequence $\bn=()$, let $f_\bn=1$ and $w(\bn)=0$.
Then $Z^\bff(s_1,\dots,s_k)=f_\bn$,
where $n_i=s_1+\dots+s_i$.
Moreover, its weight is at most
\begin{equation}
\sum_{i=1}^k \max\set{s_i,1}
=(s_1+\dots+s_k)+\#\sets{2\le i\le k}{s_i=0}=w(\bn).
\end{equation}
The weight filtrations of $\Zrp$ and $\Zr$ are given
\begin{gather}
W_n\Zrp=\spn\sets{f_\bn}{1\le n_1<\dots<n_k,\, w(\bn)\le n,\, k\ge0},
\qquad\\
W_n\Zr=\spn\sets{f_\bn}{1\le n_1\le \dots\le n_k,\, w(\bn)\le n,\, k\ge0}.
\end{gather}

\begin{theorem}
We have
\begin{enumerate}
\item
The Poincar\'e series of $\Zrp$ is
\[
P(\Zrp,t)
:=\sum_{i\ge0}\dim\gr^W_i(\Zrp)t^i=1+\sum_{i\ge1}it^i=\frac{1-t+t^2}{(1-t)^2}.\]
\item
$\Zrp$ is spanned by the elements $Z^\bff(s_1,\dots,s_k)$
with $s_i\in\set{1,2}$.
\end{enumerate}
\end{theorem}
\begin{proof}
\clm1
Let $f_k=\frac{q^k}{1-q^k}$ and $e_k=1+f_k=\frac1{1-q^k}$ for $k\ge1$.
The vector space $W_n=W_n\Zrp$ is spanned by the products of the form $\prod_{i=1}^k e_{n_i}$ for $0<n_1<\dots<n_k\le n$.
% (including the empty product, equal to $1$).
%Therefore $W_{n}$ is spanned by the products $\prod_{i=1}^l e_{n_i}$, with $n_i$ as above.
Therefore
$W_{n}=W_{n-1}+e_{n}W_{n-1}$.
Let
\[\bW_n=(q)_n\cdot W_n\sbs K[q],\qquad (q)_n=\prod*_{i=1}^n(1-q^i).\]
Then $\bW_n=(1-q^n)\bW_{n-1}+\bW_{n-1}=q^n\bW_{n-1}+\bW_{n-1}$.
Therefore
\[\bW_n=\sets{f\in K[q]}{\deg f\le \tbinom{n+1}2}\]
by induction on $n$.
This implies that $\dim W_n-\dim W_{n-1}=n$.

\clm2
Let $U_n$ be spanned by the elements $\prod_{i=1}^k e_{n_i}$
and let $V_n$ be spanned by the elements $\prod_{i=1}^k f_{n_i}$,
where
\[0=n_0<n_1<\dots<n_k\le n,\qquad n_i-n_{i-1}\in\set{1,2}.\]
Let $U_0=V_0=K$.
We have $U_n,V_n\sbs W_n$ and we need to prove that $V_n=W_n$.
Let $\ta(f)=f(q\inv)$.
Then $\ta(e_k)=-f_k$, hence $\ta(U_n)=V_n$ and $\ta(W_n)=W_n$.
Therefore it suffices to show that $U_n=W_n$.
%(then $\dim U_n=\dim V_n-\dim W_n$, hence $U_n=W_n$).

The space $\bU_n=(q)_n\cdot U_n$
contains all polynomials of degree $<n$.
Indeed, considering the sequence $[n]=(1<\dots<n)$,
we obtain $1\in\bU_n$.
Considering the sequence $[n]\ms \set k$ for $1\le k<n$,
we obtain $\frac{1-q^k}{(q)_n}\in U_n$, hence $1-q^k\in\bU_n$.
Therefore $q^k\in\bU_n$.

Every $f\in \bW_n$ can be written in the form
$f=(1-q^n)g+h$, where $g,h\in K[q]$ satisfy $\deg g\le\deg f-n\le\binom n2$ and $\deg h<n$.
Therefore $g\in \bW_{n-1}$ and $h\in\bU_n$,
and we conclude that $\bW_n=(1-q^n)\bW_{n-1}+\bU_n$.
Therefore $W_n=W_{n-1}+U_n$.
By induction, $W_{n-1}=U_{n-1}\sbs U_n$,
hence $W_n=U_n$.
\end{proof}

Now we determine the Poincar\'e series of $\Zr$.
In what follows, let $\Phi_d(q)\in\bZ[q]$ be the $d$-th cyclotomic polynomial satisfying $q^n-1=\prod_{d\mid n}\Phi_d(q)$,
and let $\phi(d)=\deg\Phi_d$ be Euler's totient function.

\def\Np{\bN_1}
\def\fl#1{\lfloor #1\rfloor}
\begin{theorem}\label{thm:preg}
We have
\begin{gather*}
\dim \gr_n^W\Zr=\sum_{d=1}^n\phi(d),\qquad n\ge1,\\
%where $\phi$ is Euler's totient function.
%Therefore
P(\Zr,t)
=\sum_{i\ge0}\dim\gr_i^W(\Zr)t^i
=1+\frac{1}{1-t}\sum_{d\ge1}\phi(d)t^d.
\end{gather*}
\end{theorem}

\begin{proof}
Let $W_n=W_n\Zr$ and $W_n^+=W_n\Zrp$.
We can assume that $K=\bQ$, since $W_n(K)=W_n(\bQ)\ts_\bQ K$.
We will show that
\[\dim W_n=1+\sum_{d=1}^n(n-d+1)\phi(d).\]
First, we claim that
\[W_n\sbs E_n=\sets{\frac{p(q)}{\De_n(q)}}{\deg p\le\deg\De_n},\qquad
\De_n=\prod_{d=1}^n\Phi_d^{n-d+1}.
\]
The vector space $W_n$ is spanned by $1$ and
$f_\bn=\prod_{i=1}^k f_{n_i}$ for $1\le n_1\le\dots\le n_k$
with $k\ge1$ and
%$w(\bn)\le n$.
\[w(\bn)=n_k+\#\setc{2\le i\le k}{n_i=n_{i-1}}\le n.\]
The exponent of $\Phi_d$ in the denominator of $f_\bn$ is
$m_d=\#\setc{i}{d\mid n_i}$.
It is zero if $d>n_k$.
Among the distinct values of $n_i$, at most $\fl{n_k/d}$ are divisible
by $d$, and repetitions contribute at most
$r=w(\bn)-n_k$.
Therefore, for $d\le n_k$,
\[m_d\le\fl{n_k/d}+r\le (n_k-d+1)+r=w(\bn)-d+1\le n-d+1.\]
Since $\deg f_\bn\le0$,
we conclude that $f_\bn=p/\De_n$ for some $p\in \bQ[q]$ with $\deg p\le\deg \De_n$.
Hence
\[\dim W_n\le \dim E_n=1+\sum_{d=1}^n(n-d+1)\phi(d).\]

Let us prove the converse inequality.
For $k\ge0$, $1\le n_1<\dots<n_k<d$ and $r\ge1$,
the element $f_d^r f_\bn$ has weight $\le d+r-1$.
Therefore $f_d^r W^+_{d-1}\sbs W_n$ for $d+r-1\le n$.
Hence
\[K+\sum_{d=1}^n \sum_{r=1}^{n-d+1}f_d^r W^+_{d-1}\sbs W_n.\]
Let us estimate the dimension of this space.
Recall that $W^+_{d-1}=\sets{\frac p{(q)_{d-1}}}{\deg p\le\binom d2}$.
Consider the localization $C_d=\bQ[q]_{(\Phi_d)}$
and let $\pi_d:C_d\to \bQ[q]/(\Phi_d)$
be the projection to the residue field.
Since $(q)_{d-1}$ is coprime with $\Phi_d$ and
$1+\binom d2\ge d\ge\deg\Phi_d$ for $d\ge2$,
the map $\pi_d:W_{d-1}^+\to \bQ[q]/(\Phi_d)$ is surjective.
For $d=1$, the map
$\pi_d:W^+_0=\bQ\to \bQ[q]/(\Phi_1)$ is also surjective.

Choose elements $g_{d,i}\in W^+_{d-1}$ for $i\in[\phi(d)]$, whose images form a basis of $\bQ[q]/(\Phi_d)$.
We claim that the elements
\[1,\,f_d^r g_{d,i},\qquad d\in[n],\,i\in[\phi(d)],\,
r\in [n-d+1],
\]
are linearly independent.
Assume there is a linear dependence $c+\sum c_{d,r,i}f_d^r g_{d,i}=0$
and choose the largest $d$ with $c_{d,r,i}\ne0$.
Note that $f_d=u_d/\Phi_d$, where $u_d$ is invertible in $\bQ[q]_{(\Phi_d)}$.
The terms indexed by $e<d$ are contained in $\bQ[q]_{(\Phi_d)}$,
since $\Phi_d$ and $1-q^m$ are coprime for $1\le m<d$.
Choose the largest $r$ with $c_{d,r,i}\ne0$.
Multiplying by $\Phi_d^r$ and reducing modulo $\Phi_d$,
we obtain
$\pi_d(u_d)^r\sum_i c_{d,r,i}\pi_d(g_{d,i})=0$,
hence $c_{d,r,i}=0$ for all $i$.
This is a contradiction.
We conclude that
\[\dim W_n\ge 1+\sum_{d=1}^n (n-d+1)\phi(d),\]
hence we have an equality.
This implies that
$\dim W_n/W_{n-1}=\sum_{d=1}^n\phi(d)$.
\end{proof}

%Literature
\bibliographystyle{halpha}
\bibliography{qzeta}
\end{document}